\documentclass[12pt]{amsart}
\usepackage[utf8]{inputenc}
\usepackage{fancyhdr}
\usepackage[dvipsnames]{xcolor}
\usepackage[all]{xy}
\usepackage{hyperref}
\usepackage{cancel}
\usepackage{tikz}

\usepackage{caption}
\usepackage{multirow}
\usepackage{placeins}
\usepackage{subcaption}

\usepackage{amsmath,amsthm,amssymb,amstext,amsbsy,amscd,amsxtra,amsfonts}
\usepackage{mathrsfs, bm, graphicx, enumerate, mathtools}
\usepackage{eucal}
\usepackage{color}
\usepackage[enableskew,vcentermath]{youngtab}

\newtheorem{thm}{Theorem}[section]
\newtheorem{theorem}{Theorem}[section]
\theoremstyle{plain}

\newtheorem{lemma}[thm]{Lemma}

\theoremstyle{definition}

\newtheorem{remark}[thm]{Remark}
\newtheorem{rem}[thm]{Remark}

\definecolor{A}{rgb}{.75,1,.75}

{\vskip-\lastskip\medskip
  \noindent
  {\em #1.}\enspace
  }%
{\qed\par\medskip
  }
\newcommand{\rarrow}{\rightarrow}

\newcommand{\disp}{\displaystyle}

\newcommand{\be}{\begin{equation}}
\newcommand{\ee}{\end{equation}}

\newcommand{\R}{\mathbb R}

\newcommand{\pa}{\partial}

\newcommand{\pt}{\frac{\partial}{\partial t}}

\newcommand\vv[1]{\mathbf{#1}}
\newcommand\ol[1]{\overline{#1}}
\newcommand{\st}{\text{ such that }}

\title[Recover All Coefficients]{Recovering All Coefficients in the Schr\"{o}dinger Equation With Finite Sets of Boundary Measurements}
\author[Shitao Liu and Antonio Pierrottet]{Shitao Liu and Antonio Pierrottet}
\address{School of Mathematical and Statistical Sciences, Clemson University, Clemson, SC 29634}
\email{liul@clemson.edu}
\email{ampierr@g.clemson.edu}
\begin{document}

\begin{abstract}
We consider an inverse problem of recovering all spatial dependent coefficients in the time dependent Schr\"odinger equation defined on an open bounded domain in $\mathbb{R}^n$, $n\geq 2$, with smooth enough boundary. We show that by appropriately selecting a finite number of initial conditions and a fixed Dirichlet boundary condition, we may recover all the coefficients in a Lipschitz stable fashion from the corresponding finitely many boundary measurements made on a portion of the boundary. The proof is based on a direct approach, which was introduced in \cite{HIY2020}, to derive the stability estimate directly from the Carleman estimates without any cut-off procedure or compactness-uniqueness argument.
\end{abstract}

\maketitle



\section{Introduction and Main Results}\label{secIntro}
% \textbf{Introduction}

In this paper we focus on the inverse problem for the Sch\"odinger equation defined on an open and bounded domain $\Omega$ in $\mathbb{R}^n$, $n\geq 2$, with smooth enough (e.g., $C^2$) boundary $\Gamma=\partial \Omega$. Specifically, consider the general Schr\"odinger equation for $w = w(t,x)$ defined on $Q^T = [0,T]\times\Omega$, $T>0$, %(we require that $T>0$ but has no sufficiently large qualifications), 
along with initial condition $w_0$ on $\Omega$ and Dirichlet boundary condition $h$ on $\Sigma^T =[0,T]\times\Gamma$ that are given in appropriate function spaces:
\begin{align}
    &\begin{cases}
        iw_t+q_2(x)\Delta w+\vv{q}_1(x)\cdot\nabla w+q_0(x)w=0&\text{in }\ Q^T\\[2mm]
        w(0,x)=w_0(x)&\text{in  }\ \Omega\\[2mm]
        w(t,x)=h(t,x)&\text{in  }\ \Sigma^T.
    \end{cases}\label{wq2PDE}
\end{align}
Here we assume the electric potential $q_0(x)$, magnetic potential $\vv{q}_1(x)$, and the coefficient $q_2(x)$ appearing in the principal part are all real-valued and satisfy
\begin{equation}\label{qReg}
\begin{cases}
    q_0=q_0(x)\in L^\infty(\Omega), \ \vv{q}_1=\vv{q}_1(x)\in \left(C^{1}(\Omega)\right)^n,\ \mbox{and} \\[2mm]
    q_2\in \mathcal{C}=\left\{q\in C^{1}(\Omega): \ \exists\ q^*>0 \st (q^*)^{-1}\leq q(x)\leq q^*, \ \forall x\in\Omega\right\}.
\end{cases}   
\end{equation}

In the following we also use the admissible set $\mathcal{A}_M$ defined as 
\begin{equation}\label{admissible}
    \mathcal{A}_M = \{f\in H^1_0(\Omega): \ \|f\|_{H^1_0(\Omega)}\leq M\}.
    \end{equation}
where $M>0$ is a fixed positive constant.
In addition, we denote $\Gamma_0$ as the open subset of $\Gamma$ where we do not have access to, and $\Gamma_1$ as the open subset of $\Gamma$ on which measurements are taken. We assume $\Gamma=\ol{\Gamma_0\bigcap\Gamma_1}$ and $\Gamma_0 \cap \Gamma_1 = \varnothing$.

We then consider the following inverse problem for the system \eqref{wq2PDE}: Recover all together the coefficients $q_0(x)$, $\vv{q}_1(x)$ and $q_2(x)$ from measurements of Neumann boundary traces of the solution $w^{(\textbf{q})} = w(w_0,h, q_2,\vv{q}_1,q_0)$ over the observed part of the boundary $\Gamma_1$ and over the time interval $[0,T]$, namely, over $\Sigma_1^T=[0,T]\times\Gamma_1$.
%Here, $T > 0$ can be chosen arbitrarily due to the instantaneous propagation of the system \eqref{wq2PDE}. For a smaller $T^2$, a proportionally larger $c_T$ must be chosen for \eqref{cT}.

To make the observed part of the boundary $\Gamma_1$ more precise, in the following we assume the following standard geometrical assumptions on the domain $\Omega$ and the unobserved part of the boundary $\Gamma_0$:

\medskip
(A1) There exists a strictly convex function $d: \overline{\Omega} \rarrow \mathbb{R}^+$ of class $C^3(\overline{\Omega})$ in the metric $g=q_2^{-1}(x)dx^2$ such that the following two properties hold true: %(through translation and re-scaling if necessary):

\smallskip
(i) The normal derivative of $d$ on the unobserved part $\Gamma_0$ of the boundary is non-positive. Namely,
\begin{equation*}
\frac{\pa d}{\pa\nu}=\langle {D}d(x), \nu(x)\rangle\leq 0, \quad  \forall x\in\Gamma_0,
\end{equation*}
where ${D}d=\nabla_g d$ for the scalar function $d$ is the gradient vector field on $\Omega$ with respect to the metric $g$, $\langle X, Y \rangle = g(X, Y)$, $\forall X, Y\in T_x\Omega$, where $T_x\Omega$ is the tangent space at $x\in\Omega$, and $\nu$ is the unit outward normal field along the boundary $\Gamma$.  

\medskip
(ii) There exists a positive constant $\rho>0$ such that the Hessian of $d$ (a second-order tensor), $D^2d$, satisfies
\begin{equation*} {D}^2d(X,X) = \langle {D}_X({D}d), X\rangle\geq \rho\langle X, X\rangle=\rho|X|_g^2, \ \forall X\in T_x\Omega\end{equation*}
where ${D}_X$ is the covariant derivative of a vector field with respect to $X$.

\medskip
(A2) %$\displaystyle\min_{x\in\overline{\Omega}}d(x)=d_0>0$, and 
$d(x)$ has no critical point on $\overline{\Omega}$. In other words,
\begin{equation*}
\inf_{x\in\overline{\Omega}}|{D}d|=p>0. % \mbox{so that we may take} \ \inf_{x\in\overline{\Omega}}\frac{|\mathcal{D}d|^2}{d}>4.
\end{equation*}
%%%

\begin{remark}\label{rem2}
(1) The geometrical assumptions above permit the construction of a vector field that enables a pseudo-convex function necessary for allowing a Carleman estimate containing no lower-order terms for the time dependent Schr\"{o}dinger equation (\ref{wq2PDE}). In fact, this applies to the general Schr\"{o}dinger equation defined on a Riemannian manifold with a general metric $g$ (see Section 2). In our present case, $\mathbb{R}^n$ becomes a Riemannian manifold with the metric $g=q_2^{-1}(x)dx^2$, so the tangent space $T_x\Omega=\mathbb{R}_x^n = \mathbb{R}^n$. As a consequence, we have ${D}d = q_2(x)\nabla d$ and hence the condition (i) is actually equivalent to the Euclidean normal derivative of $d$ on $\Gamma_0$ is non-positive since $q_2(x)$ is assumed to be strictly positive.

(2) The above type of geometrical assumptions were formulated as early as in H\"{o}rmander's classical work on partial differential operators of principal types \cite{H1963}. 
Specifically they were also formulated in %\cite{LTZ2000} under the framework of a Euclidean metric, with 
\cite{TX2007} under the general Riemannian geometry framework. For examples and detailed illustrations of large general classes of domains and boundaries $\{\Omega, \Gamma_1, \Gamma_0\}$ satisfying the aforementioned assumptions we refer to \cite[Section 10]{TX2007}. 
\end{remark}

%%%
%\textbf{Main Results}
The main result of the paper is as follows.
%Here we will directly prove the Lipschitz stability result for recovering all coefficients $\{q_2,\ \vv{q}_1,\ q_0\}$ from a finite set of boundary measurements.
% We then analyses the potential of another PDE with potentially different coefficients, $p_2$, $\vv{p}_1$, $p_0$ corresponding to the coefficients $q_2$, $\vv{q}_1$, $q_0$ in \eqref{wq2PDE}, yet still satisfying identical initial conditions and Neumann boundary trace. We call this solution $w^{(\vv{p})}$ and it satisfies the PDE,

\begin{theorem}\label{thmIP}%%% mention q regularity too? \eqref{qReg}
    Under the geometrical assumptions (A1) and (A2), let $T>0$. %let $0<c_T$ such that
%    \begin{align}
 %       \max_{\ol{\Omega}}d(x)<c_T T^2.\label{cT}
  %  \end{align}
    Suppose that the initial condition $w_0$ and boundary condition $h$ are in the following function spaces
    \begin{align}
        w_0\in H^{\gamma}(\Omega), \quad h\in H^{(\gamma+1)/2}(\Sigma^T), \quad \mbox{where} \ \gamma>\frac n2 + 4 \label{w0RegCond}
    \end{align}
    along with all compatibility conditions (trace coincidence) which make sense. In addition, suppose the following positivity condition: There exists a positive constant $r_0>0$ such that 
    \begin{align}
        r_0\leq|\text{det }W_{\bf w_0}(x)|,\ x\in\Omega\text{ a.e.}\label{Wpos}
    \end{align}
    where $W_{\bf w_0}(x)$ is the $(n+2)\times (n+2)$ matrix defined by
    \begin{align*}
    W_{\bf w_0}(x)=\begin{bmatrix}
        w_0^{(1)}  & \partial_{x_1}w_0^{(1)} & \hdots & \partial_{x_n}w_0^{(1)} & \Delta w_0^{(1)} \\[2mm]
        \vdots & \vdots  & \ddots & \vdots & \vdots\\[2mm]
        w_0^{(n+2)}  & \partial_{x_1}w_0^{(n+2)} & \hdots & \partial_{x_n}w_0^{(n+2)} & \Delta w_0^{(n+2)}
    \end{bmatrix}\label{Wr}
    \end{align*}
    in which $\{w_0^{(i)}\}_{i=1}^{n+2}$ are $n+2$ initial conditions for (\ref{wq2PDE}) that satisfy the regularity assumption in \eqref{w0RegCond}.
    
    With the initial and boundary conditions $\{w_0^{(i)},h\}$, let $w^{(\vv{q},i)}$ be the solution to \eqref{wq2PDE} with the coefficients $\{q_2,\vv{q}_1,q_0\}$ and $w^{(\vv{p},i)}$ be the solution to \eqref{wq2PDE} with the coefficients $\{p_2,\vv{p}_1,p_0\}$ respectively. Then there exists a positive constant $C>0$ which depends on $\Omega,$ $T,$ $\Gamma_1,$ $q_2,$ $\vv{q}_1,$ $q_0,$ $\{w_0^{(i)}\}_{i=1}^{n+2}$, $h$ such that
    \begin{align*}
        \|p_2-q_2\|_{L^2(\Omega)}^2+&\|\vv{p}_1-\vv{q}_1\|_{\textbf{L}^2(\Omega)}^2+\|p_0-q_0\|_{L^2(\Omega)}^2
           \leq C\sum^{n+2}_{i=1}\left\|\frac{\partial w_t^{(\vv{q},i)}}{\partial \nu}-\frac{\partial w_t^{(\vv{p},i)}}{\partial \nu}\right\|^2_{L^2(\Sigma_1^T)},
    \end{align*}
    for all such coefficients satisfying (\ref{qReg}) and $p_2-q_2, p_0-q_0\in \mathcal{A}_M$, $\vv{p}_1-\vv{q}_1\in \left(\mathcal{A}_M\right)^n$, where $\mathcal{A}_M$ is an admissible set defined in (\ref{admissible}) and $\|\cdot\|_{\textbf{L}^2(\Omega)}$ is defined as 
    $$ \|{\bf r}\|_{\textbf{L}^2(\Omega)} = \left( \int_\Omega\sum^n_{i=1}|r_i(x)|^2dx \right)^{1/2}, \ \mbox{for} \ \ {\bf r}(x) = (r_1(x), \cdots, r_n(x)).$$
\end{theorem}

\textbf{Inverse source problem.} The first step to solve the above inverse problem is to convert it into the corresponding inverse source problem. Indeed, let
\begin{equation}\label{Uconstruct}
\begin{cases}
    f_2(x)=p_2(x)-q_2(x),\
    \vv{f}_1(x)=\vv{p}_1(x)-\vv{q}_1(x), \  f_0(x)=p_0(x)-q_0(x) \\[2mm]
   R(t,x)=w^{(\mathbf{p})}(t,x), \ u(t,x)=w^{(\mathbf{q})}(t,x)-w^{(\mathbf{p})}(t,x) 
   \end{cases}
\end{equation}%%%  f\in\C, \|R_t(t,\cdot)\|\in L^1(0,T) and L^\infty. p's and q's \in L^\infty and 0<c_0\leq p_2,q_2 for some c_0\in\R

Then we have that $u=u(t,x)$ satisfies the following homogeneous mixed problem:
\begin{align}
    &\begin{cases}
        iu_t+q_2(x)\Delta u+\vv{q}_1(x)\cdot\nabla u+q_0(x)u=S(t,x)&\text{in }Q^T\\[2mm]
        u(0,x)=0&\text{in  }\Omega\\[2mm]
        u(t,x)=0&\text{in  }\Sigma^T,
    \end{cases}\label{uq2PDE}
\end{align}
where
\begin{align*}%
    &S(t,x) = f_2(x)\Delta R(t,x)+\vv{f}_1(x)\cdot\nabla R(t,x)+f_0(x)R(t,x).%\label{force}%\\
   % &\vv{f}_1\in L^2(\Omega)^n,\ f_2,\ f_0\in L^2(\Omega).\label{forceReg}%%% if we have this reg on p and the same on q then how do we get that f=p-q is L^2
\end{align*}

Here we assume $q_2$, $\vv{q}_1$ and $q_0$ are given fixed, and $R=R(t,x)$ is a given function that can be suitably chosen. The inverse source problem is to recover $f_2$, $\vv{f}_1$ and $f_0$ from the Neumann boundary measurements of $u$ over the observed part of the boundary $\Gamma_1$ and over a time interval $[0,T]$. More specifically,
corresponding to Theorem~\ref{thmIP}, we will prove the following stability result for the inverse source problem.

\begin{theorem}\label{thmISP}
    Under the geometrical assumptions (A1) and (A2). %Choosing $0<c_T$ restricted by \eqref{cT}. 
    Assume the regularity and positivity conditions on $R^{(1)},\hdots,R^{(n+2)}$ such that 
    \begin{align}
        R^{(i)}\in W^{4,\infty}(Q^T),\quad i=1,\hdots,n+2\label{RregCond}
    \end{align}
    and there exists a positive constant $r_0>0$ such that 
    \begin{align}
        r_0\leq|\text{det }U_{\bf R}(x)|,\ x\in\Omega\text{ a.e.}\label{Upos}
    \end{align}
    where $U_{\bf R}(x)$ is the $(n+2)\times(n+2)$ matrix defined by
    \begin{align}
        U_{\bf R}(x)=&\left.\begin{bmatrix}
            R^{(1)}  & \partial_{x_1}R^{(1)} & \hdots & \partial_{x_n}R^{(1)} & \Delta {R^{(1)}} \\[2mm]
            %R^{(2)} & \Delta {R^{(2)}} & \partial_{x_1} R^{(2)} & \hdots & \partial_{x_n}R^{(2)}\\
            \vdots&\vdots&\vdots&\ddots&\vdots\\[2mm]
            R^{(n+2)}  & \partial_{x_1}R^{(n+2)} & \hdots & \partial_{x_n}R^{(n+2)} & \Delta {R^{(n+2)}}
        \end{bmatrix}\right|_{t=0}\label{Ur}
    \end{align}
    Let each $u^{(i)}$ be the solution to \eqref{uq2PDE} with the function $R^{(i)}$ in the right hand side.
   % \begin{align}
      %  \vv{f}(x):=(f_0(x),f_2(x),\vv{f}_1(x))^T\label{fVec}
   % \end{align}
    Then there exists a positive constant $C>0$ that depends on $\Omega,$ $T,$ $\Gamma_1,$ $q_2,$ $\vv{q}_1,$ $q_0,$ $\{R^{(i)}\}_{i=1}^{n+2}$ such that,
    \begin{align*}
         \|f_2\|_{L^2(\Omega)}^2+&\|\vv{f}_1\|_{\textbf{L}^2(\Omega)}^2+\|f_0\|_{L^2(\Omega)}^2
        %\|f_2\|_{H^1_0(\Omega)}^2+&\|\vv{f}_1\|_{\textbf{H}^1_0(\Omega)}^2+\|f_0\|_{H^1_0(\Omega)}^2
        \leq
        C\sum^{n+2}_{i=1}\left\|\frac{\partial u_t^{(i)}}{\partial \nu}\right\|^2_{L^2(\Sigma_1^T)}
    \end{align*}
    for all $f_2,\ f_0\in \mathcal{A}_M$ and $\vv{f}_1\in \left(\mathcal{A}_M\right)^n$, where $\mathcal{A}_M$ is as defined in (\ref{admissible}).
\end{theorem}
%%%

Recovering coefficients in partial differential equations with single or finitely many boundary measurements is an important type of inverse problems and has been well studied in the literature, especially in the context of hyperbolic type equations and systems, see for instance the monographs \cite{BY2017,Isakov2000,Isakov2006,Klibanov2004,Klibanov2021,Liu-T2013} and the substantial lists of references therein. For the time dependent Schr\"{o}dinger equation considered in the present paper, it is known that one can recover the electric potential $q_0$ from a single partial boundary measurement \cite{BP2002,LiuT2011,MOR2008,TZ2015}, even when $q_0$ is complex valued \cite{HKSY2019}, or non-compactly supported \cite{BKS2016}. In the case of the vector valued magnetic potential $\vv{q}_1$, one can recover it alone or together with the electric potential from properly making finite sets of partial boundary measurements \cite{BM2018,HKSY2019,KRS2021}.

Recovering the variable unknown coefficient $q_2$ appearing in the principal part of the Schr\"{o}dinger equation is more challenging and much less studied in the literature. Nevertheless, the similar inverse problem for the wave equation plays a significant role in many practical applications and it has been well investigated \cite{B2004,Liu2013,SU2013}. The typical approach is to rewrite the equation on an appropriate Riemannian manifold so that the principal part becomes differential operators with constant coefficients. In our proof below such a step is also carried out at the beginning of Section 3.

The standard solution methods for recovering coefficients of partial differential equations from a single or finitely many boundary measurements require using the Carleman-type estimates for the underlying equations. To certain extent, all such approaches can be seen as variations or improvements of the so called Bukhgeim–Klibanov (BK) method which was originated in the seminal paper \cite{BK1981}, see also \cite{Klibanov1992}. Our approach to solve the present inverse problem uses a more recent variation of the BK method that combines a Carleman estimate for the Schr\"{o}dinger equation \cite{TX2007} and a post Carleman estimate route introduced in \cite{HIY2020}. 

Another motivation of our paper is our recent work \cite{LPS2023} where we considered the same inverse problem of recovering all coefficients in the general second-order hyperbolic equation from finite sets of boundary measurements. In particular, in \cite{LPS2023} we derived the uniqueness of all the coefficients first based on Carleman estimates, and then use the uniqueness result to achieve stability estimate based on a compactness-uniqueness argument. In the present paper, however, we employ a novel approach to directly derive the stability estimate from Carleman estimates. As mentioned above, this approach is motivated by the recent work \cite{HIY2020}. To the best of our knowledge, this is the first paper treating all the aforementioned coefficients for the Schr\"{o}dinger equation (\ref{wq2PDE}) with a direct approach to derive the stability estimate without having to use any cut-off procedures or compactness-uniqueness arguments.

The rest of the paper is organized as follows. In Section \ref{secTools} we recall some necessary tools to solve the inverse problems. This includes the Carleman estimate for the general Schr\"{o}dinger equation defined on Riemannain manifold, as well as the energy estimate and regularity theory for the Schr\"odinger equation with Dirichlet boundary condition. In Section \ref{secProofs} we provide the proofs of Theorems \ref{thmIP} and \ref{thmISP}, and in Section \ref{secExamples} we give some examples where the positivity condition \eqref{Wpos} is satisfied and some concluding remarks. We also provide a proof of the energy estimates for the Schr\"{o}dinger equation in the appendix at the end.
%%%

%%%%%%%%%%%%%%%%%%%%%%%%%%%%%%%%%%%%%%%%%%%%%%%%%%%%%

%%%
\section{Carlemen Estimate, Energy Estimate and Regularity Theory for the Schr\"odinger Equation}\label{secTools}
%%%
Consider a Riemannian manifold $(\mathcal{M},g)$ with the metric $g(\cdot, \cdot) = \langle\cdot,\cdot\rangle$ and squared norm $|X|^2=g(X,X)$. We define $\Omega$ as an open bounded, connected set of $\mathcal{M}$ with a smooth enough (e.g., $C^2$) boundary $\Gamma=\overline{\Gamma_0\bigcup\Gamma_1}$, where $\Gamma_0\bigcap\Gamma_1=\varnothing$ and $\nu$ is the unit outward normal field along the boundary $\Gamma$. On $\mathcal{M}$ we will denote $\Delta_g$ as the Laplace-Beltrami operator, and $D$ as the Levi-Civita connection.

Consider the following Schr\"odinger equation with energy level terms defined as
\begin{align}
    iw_t(t,x)+\Delta_gw(t,x)+G(w)=F(t,x),\qquad (t,x)\in Q=[-T,T]\times\Omega\label{genPDE}
\end{align}
where the forcing term $F\in L^2(Q)$ and the energy level differential term $G(w)$ is given by 
\begin{align}
    G(w)(t,x)=\langle \mathbf{P}_1(t,x),Dw(t,x)\rangle + P_0(x)w(t,x). \label{ManLoTF}
\end{align}
Here $P_0\in L^\infty(\Omega)$ is a function on $\Omega$, $\mathbf{P}_1\in C^1(Q)$ is a vector field on $\mathcal{M}$, and they satisfy the following estimate: there exists a positive constant $C_T>0$ such that 
\begin{align}
    |G(w)(t,x)|^2\leq C_T[|w(t,x)|^2+|D w(t,x)|^2],\qquad \forall (t,x)\in Q
    \label{LoTFunBound}
\end{align}
%%%

%%%
\textbf{Pseudo Convex Function}.
%In preparation for use in the Carlemen estimate, we look to constructing a weight function, 
Having chosen the function $d(x)$ under the geometrical assumptions (A1) and (A2) with the general metric $g$ here, we define the function 
$\phi(t,x):[-T,T]\times\Omega\to\R$ of the form %of $\phi(t,x)$ is chosen as
\begin{align}
    \phi(t,x):=d(x)-ct^2,\label{phi}
\end{align}
where $T>0$ is an arbitrary positive constant, and $c$ is chosen such that 
\begin{equation*}\label{cT}
    d_1:=\max_{x\in\overline{\Omega}}d(x) < cT^2
\end{equation*}
so that for some suitable small fixed $\delta>0$
\begin{equation*}\label{delta}
d_1 +\delta < cT^2.
\end{equation*}
%where we define $d_1:=\|d\|_{C(\ol{\Omega})}$. 
% $\phi$ inherits the class of $C^3$ from $d(x)$ by (A1). From \eqref{cT} we can chose a $\delta$ small enough\st 
%\begin{align}
   % d_1+\delta<c_TT^2.\label{delta}
%\end{align}
From here on we keep $\phi$ fixed with the choice of $d(x),\ T,\ c$ and $\delta$. Under the above conditions $\phi$ can be seen to have the following properties,
\begin{align}
    \phi(-T,x)&=\phi(T,x)\leq  d_1-cT^2<-\delta&&x\in{\Omega}\label{phi_delta}
    \intertext{and}
    \phi(t,x)&=d(x)-ct^2\leq d(x) = \phi(0,x)&&x\in\Omega. \nonumber
\end{align}
\textbf{Carleman estimates}. We now come back to the equation (\ref{genPDE}), and consider the solutions $w=w(t,x)$ in the class
\begin{equation}\label{wclass}
w\in C([-T,T]; H^1(\Omega)), \quad w_t, \ \langle Dw, \nu\rangle\in L^2(-T,T; L^2(\Gamma)).
\end{equation}
Then for these solutions we have the following Carleman estimates as in \cite[Theorem~2.1]{TX2007}.
\begin{theorem}\label{thmCE}%#%#% Triggiani \& Xu Theorem 2.2.1
   Under the geometrical assumptions (A1) and (A2) with respect to the general metric $g$. Let $T>0$ and $\phi$ be as defined above. Let $w$ be a solution of (\ref{genPDE}) in the class (\ref{wclass}) and under the assumptions (\ref{ManLoTF}) and (\ref{LoTFunBound}). Then the following one-parameter family of estimates hold true for all sufficiently large $\tau > 0$:
    \begin{align}
        &\left[\delta_0\left(2\tau\rho-\frac12\right)-4C_T\right]\int_{Q} e^{2\tau \phi}|D w|^2dQ
        \nonumber\\&
        +\left[4\tau^3\rho p^2(1-\delta_0)+O(\tau^2)-4C_T\right]\int_{Q}e^{2\tau \phi}\left|w\right|^2dQ
        \nonumber\\ \leq&
        BT(w)+4\int_{Q}e^{2\tau \phi}\left|F\right|^2dQ+C_{d,T}\tau e^{-2\tau\delta} 
                \left[\mathbb{E}_w(T)+\mathbb{E}_w(-T)\right]\label{carleman}%\int_{\Omega}\left|D w(-T,x)\right|^2 + \left|w(-T,x)\right|^2 + \left|D w(t,x)\right|^2 + \left|w(t,x)\right|^2 d\Omega
    \end{align}
    where $\delta_0\in(0,1)$ is a constant and the energy function $\mathbb{E}_w$ is defined as 
    \begin{align}
        \mathbb{E}_w(t)=&\int_{Q} \left|w(t,x)\right|^2+|D w(t,x)|^2dQ \label{energy}
    \end{align}
    and the boundary term $BT(w)$ is given by (here $\Sigma = [-T,T]\times\Gamma$)
    \begin{align*}
        BT(w)=& \ 4\tau^3\int_\Sigma e^{2\tau\phi}\langle Dd,\nu \rangle |Dd|^2|w|^2d\Sigma
        \\&
        -4c\tau\int_\Sigma e^{2\tau\phi}t\left[\text{Im}(w)\langle D\text{Re}(w),\nu \rangle-\text{Re}(w)\langle D\text{Im}(w),\nu \rangle\right]d\Sigma
        \\&
        -2\tau\int_\Sigma e^{2\tau\phi}\langle Dd,\nu \rangle\left[\text{Re}(w)_t\text{Im}(w)-\text{Re}(w)\text{Im}(w)_t\right]d\Sigma
        \\&
        +\int_\Sigma e^{2\tau\phi}\left[2\tau^2|Dd|^2-\tau\Delta d\right]\left[\ol{w}\langle Dw,\nu \rangle+w\langle D\ol{w},\nu \rangle\right]d\Sigma
        \\&
        +2\tau\int_\Sigma e^{2\tau\phi}\left\langle Dd,D\ol{w}\langle Dw,\nu \rangle+Dw\langle D\ol{w},\nu \rangle \right\rangle d\Sigma
        \\&
        -2\tau\int_\Sigma e^{2\tau\phi}\langle Dd,\nu \rangle |Dw|^2d\Sigma
    \end{align*}
  where $\text{Re}(w)$ and $\text{Im}(w)$ denote the real and imaginary part of $w$.   
\end{theorem}
\begin{rem}
If we have that $w|_\Sigma\equiv0$, %$\partial \Omega$ (so $|Dw|=\left|\frac{\partial w}{\partial \nu}\right|$) and $\left.\left\langle Dd,D\ol{w}Dw\right\rangle\right|_{\Sigma_0}\leq0$ 
then in view of the geometrical assumption (A1) we have
\begin{equation}\label{simpleBT}
    BT(w) = 2\tau\int_{\Sigma}e^{2\tau \phi}\langle Dd, \nu\rangle|Dw|^2d\Sigma \leq 2\tau\int_{\Sigma_1} e^{2\tau \phi}\langle Dd, \nu\rangle|Dw|^2d\Sigma.
\end{equation}
\end{rem}
 %%%

%%%
\textbf{Energy estimate and regularity theory for general Schr\"{o}dinger equations with Dirichlet boundary condition}. 
Consider the Schr\"{o}dinger equation (\ref{genPDE}) with initial condition $w(0,x)=w_0(x)$ and Dirichlet boundary condition $w(t,x)=h(t,x)$ on $\Sigma$. Moreover, assume $G$ is as defined in (\ref{ManLoTF}) and satisfies (\ref{LoTFunBound}). Then the following regularity result for the solution $w$ holds true (this may be derived by interpolating between Theorems 10.9.10.1 and 10.9.7.1 in \cite{LTBook}): For $\gamma\geq -1$, suppose 
\begin{equation*}
        w_0\in H^{\gamma}(\Omega),\ \
        h\in H^{(\gamma+1)/2}\left(\Sigma\right),\ \
        F\in L^1\left([-T,T]; H^{\gamma}(\Omega)\right)
    \end{equation*}
with all compatibility conditions (trace coincidence) which makes sense. Then we have 
\begin{equation}\label{wReg}
        w\in C\left([-T,T]; H^{\gamma}(\Omega)\right).
    \end{equation}

In addition, for the equation (\ref{genPDE}), if we have the homogeneous boundary condition $h(t,x)=0$ and $w_0\in H^1_0(\Omega)$, $F\in H^1(-T,T; L^2(\Omega))$, $\vv{P}_1(t,\cdot)\in \left(C^1(\Omega)\right)^n$ for all $t\in[-T,T]$, and $P_0\in L^\infty(\Omega)$, then the following {\it a-priori} estimate holds true for the solution $w$: 
\begin{equation}\label{energyestimate}
\mathbb{E}_w(t)\leq C \left(\|w_0\|^2_{H^1_0(\Omega)}+\|F\|^2_{H^{1}(-T,T;L^2(\Omega))}\right),
    \end{equation}
where $C$ depends on $\Omega,\ T,\ \vv{P}_1,$ and $P_0$, and $\mathbb{E}_w(t)$ is defined as in (\ref{energy}) (which is equivalent to $\int_{Q} |D w(t,x)|^2dQ$ by Poincar\'{e} inequality). In the appendix we will provide a proof of the above estimate (\ref{energyestimate}).

\section{Main Proofs}\label{secProofs}
In this section we give the main proofs of the uniqueness and stability results established in Section \ref{secIntro}. We focus on proving Theorem \ref{thmISP} for the inverse source problem since Theorem \ref{thmIP} of the original inverse problem will then follow from the relation \eqref{Uconstruct} between the two problems and the regularity theory result recalled in Section \ref{secTools}. Henceforth for convenience we use $C$ to denote a generic positive constant which may depend on $\Omega$, $T$, $q_2$, $\vv{q}_1$, $q_0$, $r_0$, $w^{(i)}_0$, $R^{(i)}$, $i = 1,... ,n+2$, but not on the free large parameter $\tau$ appearing in the Carleman estimate.
%%%

\medskip
{\bf Proof of Theorem \ref{thmISP}}.
%\textbf{Choice of Manifold Representation}\\
We return to the $u$-equation (\ref{uq2PDE}). Corresponding with the choice $R^{(i)}$, $i=1,\cdots,n+2$, we have $n+2$ equations of the form (\ref{uq2PDE}). 
Extend $u^{(i)}$ and $R^{(i)}$, $i=1,\cdots, n+2$, to $(-T,0)$ by 
\begin{equation}\label{extension}
u^{(i)}(x,t) = \overline{u^{(i)}(x,-t)}, \quad R^{(i)}(x,t) = \overline{R^{(i)}(x,-t)}, \ t\in[-T,0) 
\end{equation}
then we get the solution $u^{(i)} = u^{(i)}(x,t)$ satisfies
\begin{align}
    &\begin{cases}
        iu_t^{(i)}+q_2(x)\Delta u^{(i)}+\vv{q}_1(x)\cdot\nabla u^{(i)}+q_0(x)u^{(i)}=S^{(i)}(t,x)&\text{in }Q=[-T,T]\times\Omega\\[2mm]
        u^{(i)}(0,x)=0&\text{in  }\Omega\\[2mm]
        u^{(i)}(t,x)=0&\text{in  }\Sigma=[-T,T]\times\Gamma
    \end{cases}\label{uq2PDEnew}
\end{align}
where again we denote for $t\in[-T,T]$ and $x\in\Omega$,
\begin{align*}%
    &S^{(i)}(t,x) = f_2(x)\Delta R^{(i)}(t,x)+\vv{f}_1(x)\cdot\nabla R^{(i)}(t,x)+f_0(x)R^{(i)}(t,x).%\label{forcenew}%\\
   % &\vv{f}_1\in L^2(\Omega)^n,\ f_2,\ f_0\in L^2(\Omega).\label{forceReg}%%% if we have this reg on p and the same on q then how do we get that f=p-q is L^2
\end{align*}

Differentiate the above system in time $t$, we get $u_t^{(i)}$ satisfies the equation
\begin{align}
    &\begin{cases}
        i(u_t^{(i)})_t+q_2(x)\Delta u_t^{(i)}+\vv{q}_1(x)\cdot\nabla u_t^{(i)}+q_0(x)u_t^{(i)}=S_t^{(i)}(t,x)&\text{in }Q\\[2mm]
        u_t^{(i)}(0,x)=-i S^{(i)}(0,x)&\text{in  }\Omega\\[2mm]
        u_t^{(i)}(t,x)=0&\text{in  }\Sigma.
    \end{cases}\label{utPDE}
\end{align}

Note in (\ref{uq2PDEnew}), plug in $t=0$ we may get for $i=1,\cdots, n+2$,
\begin{equation*}\label{t=0}
iu_t^{(i)}(0,x) = S^{(i)}(0,x) = f_2(x)\Delta R^{(i)}(0,x)+\vv{f}_1(x)\cdot\nabla R^{(i)}(0,x)+f_0(x)R^{(i)}(0,x).
\end{equation*} 
Put all of these $n+2$ equations together we have a linear system 
\begin{equation*}
i\left(u_t^{(1)}(0,x), \cdots, u_t^{(n+2)}(0,x)\right)^T = U_{\bf R}(x)\left(f_0(x), \vv{f}_1(x), f_2(x)\right)^T
\end{equation*}
where $U_{\bf R}(x)$ is the $(n+2)\times(n+2)$ coefficient matrix as defined in (\ref{Ur}). By the positivity assumption (\ref{Upos}) we may invert $U_{\bf R}(x)$ to obtain 
\begin{equation}\label{32}
|f_2(x)|^2+|\vv{f}_1(x)|^2+|f_0(x)|^2 \leq C \sum_{i=1}^{n+2}|u_t^{(i)}(0,x)|^2 = C\|{\bf u}_t(0,x)\|^2
\end{equation}
where we denote ${\bf u}(t,x) = \left(u^{(1)}(t,x), \cdots, u^{(n+2)}(x,t)\right)\in\mathbb{C}^{n+2}$ and $\|\cdot\|$ the standard Euclidean vector norm in $\mathbb{C}^{n+2}$. In addition, note we have the following exponentially weighted identity for the term on the right-hand side of (\ref{32}).
\begin{lemma}
For ${\bf u}(t,x) = \left(u^{(1)}(t,x), \cdots, u^{(n+2)}(x,t)\right)$ where $u^{(i)}(t,x)$ satisfies (\ref{uq2PDEnew}), the following identity holds true:
\begin{eqnarray}\label{33}
& \ & \int_{\Omega} e^{2\tau\phi(0,x)}\|{\bf u}_t(0,x)\|^2\,dx \nonumber\\[2mm]
 & = & -4c\tau\int_{\Omega}\int_{-T}^0te^{2\tau\phi}\|{\bf u}_t\|^2\,dtdx + \int_{\Omega}e^{2\tau\phi(-T,x)}\|{\bf u}_t(-T,x)\|^2\,dx \nonumber \\[2mm]
 & \ & +2 \, Im\int_{\Omega}\int_{-T}^0  \left(J_{{\bf u}_t}\left[\nabla(e^{2\tau\phi}q_2) - e^{2\tau\phi}\vv{q}_1\right]\right)\cdot \overline{{\bf u}_t}\,dtdx \nonumber \\[2mm]
 & \ & + 2 \, Im\int_{\Omega}\int_{-T}^0 e^{2\tau\phi} \overline{{\bf u}_t}\cdot {\bf S}_t \,dtdx
\end{eqnarray}
where $\phi$ is the psudeo-convex function as defined in (\ref{phi}), $J_{{\bf u}_t}$ is the Jacobian matrix of ${\bf u}_t$, ${\bf S} = \left(S^{(1)}(t,x), \cdots, S^{(n+2)}(x,t)\right)\in\mathbb{C}^{n+2}$, and $\cdot$ denotes the standard dot product for complex vectors. 
\end{lemma}
\begin{proof}
Note for each $i=1, \cdots, n+2$, we have
\begin{eqnarray}\label{34}
& \ & \int_{\Omega}e^{2\tau\phi(0,x)}|u_t^{(i)}(0,x)|^2\,dx \nonumber \\[2mm] & = & \int_{\Omega}\int_{-T}^0\frac{d}{dt}(e^{2\tau\phi}|u_t^{(i)}|^2)\,dtdx + \int_{\Omega}e^{2\tau\phi(-T,x)}|u_t^{(i)}(-T,x)|^2\,dx \nonumber \\[2mm]
& = & -4c\tau\int_{\Omega}\int_{-T}^0 te^{2\tau\phi}|u_t^{(i)}|^2\,dtdx+2Re\int_{\Omega}\int_{-T}^0e^{2\tau\phi}\overline{u_t^{(i)}}u_{tt}^{(i)}\,dtdx \nonumber \\
& \ & \ +\int_{\Omega}e^{2\tau\phi(-T,x)}|u_t^{(i)}(-T,x)|^2\,dx
\end{eqnarray}
Evaluate the second integral term on the right-hand side of (\ref{34}), note by using the $u_t^{(i)}$-equation from (\ref{utPDE}), we have 
\begin{eqnarray}\label{35}
& \ & 2Re\int_{\Omega}\int_{-T}^0e^{2\tau\phi}\overline{u_t^{(i)}}u_{tt}^{(i)}\,dtdx \nonumber\\[2mm]
& = & -2Im \int_{\Omega}\int_{-T}^0e^{2\tau\phi}\overline{u_t^{(i)}}\left[q_2(x)\Delta u_t^{(i)}+\vv{q}_1(x)\cdot\nabla u_t^{(i)}+q_0(x)u_t^{(i)}-S_t^{(i)}\right]\,dtdx \nonumber\\[2mm]
& = & -2Im \int_{\Omega}\int_{-T}^0e^{2\tau\phi}\overline{u_t^{(i)}}\left[q_2(x)\Delta u_t^{(i)}+\vv{q}_1(x)\cdot\nabla u_t^{(i)}-S_t^{(i)}\right]\,dtdx 
\end{eqnarray}
where in the last step we use the fact $q_0(x)|u_{t}^{(i)}|^2\in\mathbb{R}$. Next, by the Green's formula, the vanishing boundary condition $u_t^{(i)}=0$ on $\Sigma$, and the fact $q_2(x)|\nabla u_{t}^{(i)}|^2\in\mathbb{R}$, we may evaluate the first integral term on the right-hand side of (\ref{35}) as
\begin{eqnarray}\label{36}
-2Im\int_{\Omega}\int_{-T}^0e^{2\tau\phi}\overline{u_t^{(i)}}q_2(x)\Delta u_t^{(i)}\,dtdx & = & 2Im\int_{\Omega}\int_{-T}^0 \nabla (e^{2\tau\phi}q_2(x)\overline{u_t^{(i)}})\cdot\nabla u_t^{(i)}\,dtdx \nonumber\\[2mm]
& = & 2Im\int_{\Omega}\int_{-T}^0 \nabla (e^{2\tau\phi}q_2(x))\cdot\nabla u_t^{(i)} \overline{u_t^{(i)}}\,dtdx.
\end{eqnarray}
Combine together (\ref{34}), (\ref{35}) and (\ref{36}), we arrive at 
\begin{eqnarray*}\label{37}
& \ & \int_{\Omega}e^{2\tau\phi(0,x)}|u_t^{(i)}(0,x)|^2\,dx \nonumber \\[2mm] 
& = & -4c\tau\int_{\Omega}\int_{-T}^0 te^{2\tau\phi}|u_t^{(i)}|^2\,dtdx+\int_{\Omega}e^{2\tau\phi(-T,x)}|u_t^{(i)}(-T,x)|^2\,dx \nonumber\\[2mm]
& \ & \ + 2Im\int_{\Omega}\int_{-T}^0 \left[\nabla (e^{2\tau\phi}q_2) - e^{2\tau\phi}\vv{q}_1\right]\cdot \nabla u_t^{(i)}\overline{u_t^{(i)}}\,dtdx + 2Im\int_{\Omega}\int_{-T}^0e^{2\tau\phi}\overline{u_t^{(i)}}S_t^{(i)}\,dtdx.
\end{eqnarray*}
Sum the above identity over $i=1,\cdots, n+2$ we readily get the desired identity (\ref{33}).
\end{proof}

To continue, note by the regularity assumptions on $q_2$ and $\vv{q}_1$ in (\ref{qReg}), the above identity (\ref{33}) gives the following estimate
\begin{eqnarray}\label{38}
& \ & \int_{\Omega} e^{2\tau\phi(0,x)}\|{\bf u}_t(0,x)\|^2\,dx \nonumber \\[2mm]
& \leq & C\tau\int_{\Omega}\int_{-T}^0-te^{2\tau\phi}\left(\|{\bf u}_t\|^2+\|\nabla{\bf u}_t\|^2\right)\,dtdx \nonumber \\
& \ & +C \int_{\Omega}\int_{-T}^0-te^{2\tau\phi}\|{\bf S}_t\|^2\,dtdx+C\int_{\Omega}e^{2\tau\phi(-T,x)}\|{\bf u}_t(-T,x)\|^2\,dx \nonumber \\[2mm]
& \leq & C\tau\int_{Q}e^{2\tau\phi}\left(\|{\bf u}_t\|^2+\|\nabla{\bf u}_t\|^2\right)\,dQ \nonumber \\
& \ & +C \int_{Q}e^{2\tau\phi}\|{\bf S}_t\|^2\,dQ+C\int_{\Omega}e^{2\tau\phi(-T,x)}\|{\bf u}_t(-T,x)\|^2\,dx 
\end{eqnarray}
where $\|\nabla {\bf u}_t\|^2 = \disp\sum_{i=1}^{n+2}|\nabla u_t^{(i)}|^2$. 

We now estimate the terms on the right hand side of (\ref{38}). First, note since $q_2\in\mathcal{C}$, the above $u_t^{(i)}$ equations (\ref{utPDE}) can be written as Schr\"{o}dinger equations on the Riemannian manifold with the metric $g=q_2^{-1}(x)dx^2$, module lower-order terms as follows:
\begin{equation*}\label{riemannut}
i(u_t^{(i)})_t+\Delta_g u_t^{(i)}+q_2^{-1}\left(\vv{q}_1(x)-q_2^{\frac{n}{2}}\nabla(q_2^{\frac{2-n}{2}})\right)\cdot\nabla_g u_t^{(i)}+q_0(x)u_t^{(i)}=S_t^{(i)}(t,x).
\end{equation*}  
By the regularity assumption (\ref{RregCond}) on $R^{(i)}$, we have $S_t^{(i)}\in L^2(Q)$, $S(0,x)\in H^1_0(\Omega)$ and thus we may apply the Carleman estimate (\ref{carleman}) for the solutions $u_t^{(i)}$,  $i=1,\cdots, n+2$, and get the following inequality for sufficiently large $\tau$: 
\begin{eqnarray*}\label{40}
& \ & \tau\int_{Q}e^{2\tau\phi}|Du_t^{(i)}|^2\,dQ + \tau^3\int_Qe^{2\tau\phi}|u_t^{(i)}|^2\,dQ \nonumber \\[2mm]
& \leq & BT(u_t^{(i)})+C \int_Qe^{2\tau\phi}|S_t^{(i)}|^2\,dQ+C\tau e^{-2\tau\delta}\left[\mathbb{E}_{u_t^{(i)}}(T)+\mathbb{E}_{u_t^{(i)}}(-T)\right].
\end{eqnarray*}
Note $Du_t^{(i)} = q_2(x)\nabla u_t^{(i)}$ and $q_2(x)\geq (q^*)^{-1}$ from (\ref{qReg}), we may add the above $n+2$ inequalities together to get
\begin{eqnarray}\label{41}
& \ & \tau\int_{Q}e^{2\tau\phi}\|\nabla {\bf u}_t\|^2\,dQ + \tau^3\int_Q\|{\bf u}_t\|^2\,dQ \nonumber \\[2mm]
& \leq & BT({\bf u}_t)+C \int_Qe^{2\tau\phi}\|{\bf S}_t\|^2\,dQ+C\tau e^{-2\tau\delta}\left[\mathbb{E}_{{\bf u}_t}(T)+\mathbb{E}_{{\bf u}_t}(-T)\right].
\end{eqnarray}
Plug this into (\ref{38}), combine with (\ref{32}) and use the fact that $\phi(-T,x)<-\delta$ from (\ref{phi_delta}), we get
\begin{eqnarray}\label{42}
& \ & \int_{\Omega} e^{2\tau\phi(0,x)} \left[|f_2(x)|^2+|\vv{f}_1(x)|^2+|f_0(x)|^2\right]\,dx \nonumber \\[2mm] 
& \leq &BT({\bf u}_t)+C \int_Qe^{2\tau\phi}\|{\bf S}_t\|^2\,dQ +C\tau e^{-2\tau\delta}\left[\mathbb{E}_{{\bf u}_t}(T)+\mathbb{E}_{{\bf u}_t}(-T)\right].
\end{eqnarray}

For the middle term on the right-hand side of (\ref{42}), we claim we have 
\begin{eqnarray}\label{44}
\int_Q e^{2\tau\phi}\|{\bf S}_t\|^2\,dQ = o(1)\int_{\Omega} e^{2\tau\phi(0,x)}\left[|f_2(x)|^2+|\vv{f}_1(x)|^2+|f_0(x)|^2\right]\,dx
\end{eqnarray}
where the term $o(1)$ satisfies $\displaystyle\lim_{\tau\to\infty} o(1)=0$. To prove this claim, note for $i=1,\cdots, n+2$, we have
\begin{equation*}
    S^{(i)}_t(t,x) = f_2(x)\Delta R^{(i)}_t(t,x)+\vv{f}_1(x)\cdot\nabla R^{(i)}_t(t,x)+f_0(x)R^{(i)}_t(t,x)
\end{equation*}
and we may estimate the term $\int_{Q}e^{2\tau\phi}|f_2|^2||\Delta R^{(i)}_t|^2\,dtdx$ as follows 
\begin{eqnarray}\label{45}
\int_{Q}e^{2\tau\phi}|f_2|^2||\Delta R^{(i)}_t|^2\,dtdx & = & \int_Q e^{2\tau\phi(0,x)}|f_2|^2e^{2\tau[\phi(t,x)-\phi(0,x)]}|\Delta R^{(i)}_t|^2\,dtdx \nonumber \\[2mm]
& \leq & \int_{\Omega}e^{2\tau\phi(0,x)}|f_2|^2\left(\int_{-T}^Te^{-2c\tau t^2}\|\Delta R_t^{(i)}\|_{L^\infty(\Omega)}^2\,dt\right)dx.
\end{eqnarray}
Note $e^{-2c\tau t^2}\to 0$ as $\tau\to\infty$ except at $t=0$, and from assumption (\ref{RregCond}) we have  $\Delta R_t^{(i)}\in L^1(Q)$. Hence by Dominated Convergence Theorem we have
\begin{equation*}\label{46}
\lim_{\tau\to\infty}\int_{-T}^T e^{-2c\tau t^2}\|\Delta R_t^{(i)}\|_{L^\infty(\Omega)}^2\,dt = 0.
\end{equation*}
Plug this into (\ref{45}), we thus get 
\begin{equation*}
\int_Q e^{2\tau\phi}|f_2|^2||\Delta R^{(i)}_t|^2\,dQ = o(1)\int_{\Omega} e^{2\tau\phi(0,x)}|f_2(x)|^2]\,dx.
\end{equation*}
Apply similar estimates to $\int_{Q}e^{2\tau\phi}|\vv{f}_1|^2||\nabla R^{(i)}_t|^2\,dtdx$ and $\int_{Q}e^{2\tau\phi}|f_0|^2||R^{(i)}_t|^2\,dtdx$, then sum over $i=1,\cdots, n+2$, we readily get the desired estimate (\ref{44}).

Next, first the boundary term on the right-hand side of (\ref{42}), notice in view of (\ref{simpleBT}) and the homogeneous boundary conditions $u^{(i)}\equiv0$ on $\Sigma$ and $Du^{(i)} = q_2(x)\nabla u^{(i)}$, $i=1,\cdots, n+2$, we have the following estimate (recall $d_1= \disp\max_{x\in\overline{\Omega}} d(x)$)
\begin{equation}\label{45'}
BT({\bf u}_t) \leq C\tau e^{2\tau d_1}\left\|D{\bf u}_t\right\|^2_{L^2(\Sigma_1)} \leq C\tau e^{2\tau d_1}\left\|\frac{\partial {\bf u}_t}{\partial \nu}\right\|^2_{L^2(\Sigma_1)}. %= C e^{C\tau}\sum^{n+2}_{i=1}\left\|\frac{\partial u_t^{(i)}}{\partial \nu}\right\|^2_{L^2(\Sigma_1)}.
\end{equation}

Finally, note that $\phi(x,0)=d(x)$ and $\displaystyle d_0 := \min_{x\in\overline{\Omega}} d(x) > 0$. Thus by plugging (\ref{44}) and (\ref{45'}) in (\ref{42}), we get
\begin{eqnarray*}
& \ & \int_{\Omega} \left[|f_2(x)|^2+|\vv{f}_1(x)|^2+|f_0(x)|^2\right]\,dx \\[2mm]
& \leq & C\tau e^{2\tau(d_1-d_0)}\left\|\frac{\partial {\bf u}_t}{\partial \nu}\right\|^2_{L^2(\Sigma_1)}+C\tau e^{-2\tau(\delta+d_0)}\left[\mathbb{E}_{{\bf u}_t}(T)+\mathbb{E}_{{\bf u}_t}(-T)\right].
\end{eqnarray*}

Note by applying the energy estimate (\ref{energyestimate}) to the $u_t^{(i)}$ equation (\ref{utPDE}), sum over $i=1,\cdots, n+2$, and applying the regularity assumption (\ref{RregCond}) as well as the Poincar\'{e} inequality, we have
\begin{eqnarray*}
\mathbb{E}_{{\bf u}_t}(T)+\mathbb{E}_{{\bf u}_t}(-T) & \leq & C \left(\|{\bf S}_t\|^2_{H^1(-T,T; L^2(\Omega))}+\|{\bf S}(0,x)\|_{H^1_0(\Omega)}^2\right) \nonumber \\[2mm]
& \leq & C\left[\|f_2\|_{H_0^1(\Omega)}^2+\|\vv{f}_1\|_{{\bf H}_0^1(\Omega)}^2+\|f_0\|_{H_0^1(\Omega)}^2\right]
\end{eqnarray*} 
where for ${\bf r}(x) = (r_1(x), \cdots, r_n(x))$, $\disp\|{\bf r}\|_{{\bf H}_0^1(\Omega)}^2 = \int_\Omega\sum^n_{i=1}|\nabla r_i(x)|^2dx$.
As $f_2, f_0\in\mathcal{A}_M$ and $\vv{f}_1\in (\mathcal{A}_M)^n$, we hence get 
\begin{equation*}
\int_{\Omega} \left[|f_2(x)|^2+|\vv{f}_1(x)|^2+|f_0(x)|^2\right]\,dx \leq C\tau e^{2\tau(d_1-d_0)}\left\|\frac{\partial {\bf u}_t}{\partial \nu}\right\|^2_{L^2(\Sigma_1)}+CM^2\tau e^{-2\tau(\delta+d_0)}.
\end{equation*}
As $\disp\tau e^{-2\tau(\delta+d_0)}\to 0$ and $\tau e^{2\tau(d_1-d_0)}\to\infty$, $\tau\to\infty$, we get for $\tau>0$ large enough, 
\begin{equation*}
\int_{\Omega} \left[|f_2(x)|^2+|\vv{f}_1(x)|^2+|f_0(x)|^2\right]\,dx \leq C\tau e^{2\tau(d_1-d_0)}\left\|\frac{\partial {\bf u}_t}{\partial \nu}\right\|^2_{L^2(\Sigma_1)}.
\end{equation*}
%Finally, note that $\phi(x,0)=d(x)$, and $\displaystyle \min_{x\in\overline{\Omega}} d(x) > 0$. Thus in view of (\ref{43}) and (\ref{44}), all the terms on the right-hand side of (\ref{41}), except the boundary terms, can be absorbed into the left-hand side %by $e^{2\tau \min d(x)}\left( \|f_2\|_{L^2(\Omega)}^2+\|\vv{f}_1\|_{\textbf{L}^2(\Omega)}^2+\|f_0\|_{L^2(\Omega)}^2\right)$ 
%when $\tau$ is taken large enough. Putting the above together we hence get
Namely, for sufficiently large $\tau$,
\begin{equation*}
\|f_2\|_{L^2(\Omega)}^2+\|\vv{f}_1\|_{{\bf L}^2(\Omega)}^2+\|f_0\|_{L^2(\Omega)}^2 \leq C\tau e^{C\tau}\left\|\frac{\partial {\bf u}_t}{\partial \nu}\right\|^2_{L^2(\Sigma_1)} = C e^{C\tau}\sum^{n+2}_{i=1}\left\|\frac{\partial u_t^{(i)}}{\partial \nu}\right\|^2_{L^2(\Sigma_1)}.
\end{equation*}
Since $u_t^{(i)} = -\overline{u^{(i)}(x,-t)}, t\in (-T,0)$ from the extension (\ref{extension}), we get $\left\|\frac{\partial u_t^{(i)}}{\partial \nu}\right\|^2_{L^2(\Sigma_1)}=2\left\|\frac{\partial u_t^{(i)}}{\partial \nu}\right\|^2_{L^2(\Sigma_1^T)}$, and hence the desired stability estimate: There exists a constant $C=C(\Omega, T, \Gamma_1, q_2, \vv{q}_1, q_0, \{R^{(i)}\}_{i=1}^{n+2})>0$, %which depends on $\tau$ exponentially,    %\note{this constant $C$ depends on $\tau$!} 
such that  
\begin{equation*} 
  \|f_2\|_{L^2(\Omega)}^2+\|\vv{f}_1\|_{\textbf{L}^2(\Omega)}^2+\|f_0\|_{L^2(\Omega)}^2
        \leq  C\sum^{n+2}_{i=1}\left\|\frac{\partial u_t^{i)}}{\partial \nu}\right\|^2_{L^2(\Sigma_1^T)}.
\end{equation*}

{\bf Proof of Theorem \ref{thmIP}}.
Now we provide the proof of stability for the original inverse problem. These results are a direct consequence of Theorem \ref{thmISP} with the given relationship \eqref{Uconstruct} between the original inverse problem and the inverse source problem. More precisely, we have the positivity condition \eqref{Wpos} imply \eqref{Upos}. In addition, by the regularity result (\ref{wReg}), the assumption \eqref{w0RegCond} on the initial boundary conditions $\{w_0^{(i)},h\}$ implies the solutions $w^{(i)}$, $i=1,\cdots,n+2$, satisfy
\begin{equation*}
w^{(i)} \in C([0,T]; H^{\gamma}(\Omega)).
\end{equation*}
As $\gamma>\frac{n}{2}+4$, we have the following embedding $H^{\gamma}(\Omega)\mapsto W^{4,\infty}(\Omega)$ and hence the regularity assumption \eqref{w0RegCond} implies the corresponding regularity assumption \eqref{RregCond} for the inverse source problem. This completes the proof of all the theorems.

\section{Some Examples and Concluding Remarks}\label{secExamples}

In this last section we first provide some concrete examples such that the key positivity condition \eqref{Wpos} is satisfied, and then we give some general remarks.
%Again, as in Remark~\ref{rem1.2} we note that although the damping coefficient $q_1(x)$ appears in the matrices $W(x)$, $\widetilde{W}(x)$, $U(x)$ and $\widetilde{U}(x)$, the determinants of those matrices are actually independent of $q_1$ through elementary row operations.
 
\medskip
\noindent{\bf Example 1}. Consider the following functions $w_0^{(i)}(x)$, ${x}=(x_1, \cdots, x_n)\in\Omega$, $i=1, \cdots, n+2$, defined by
\begin{equation*}
\begin{cases}
w_0^{(1)}(x) =  1, \ w_0^{(2)}(x)  = x_1, \ w_0^{(n+2)}(x)  =  \frac{1}{2}x_1^2 \\[2mm] 
w_0^{(i)}(x) = x_{i-1}, \ i=3,\cdots, n+1.
\end{cases}
\end{equation*} 
Then we may easily see that the matrix $W_{\bf w_0}(x)$ is a lower triangular matrix with all $1$'s along the diagonal. Thus the determinant of the $W_{\bf w_0}(x)$ matrix is $1$ for all $x\in\Omega$ and therefore the positivity condition \eqref{Wpos} is satisfied. 

\medskip
\noindent{\bf Example 2}. Consider the following functions $w_0^{(i)}(x)$, $x=(x_1, \cdots, x_n)\in\Omega$, $i=1, \cdots, n+2$, defined by
\begin{equation*}
w_0^{(1)}(x) = 1, \ w_0^{(i)}(x) = e^{x_{i-1}}, \ i=2,\cdots, n+1, \ w_0^{(n+2)}(x) = e^{-x_{1}}.
\end{equation*}
Then the matrix $W_{\bf w_0}(x)$ becomes
\begin{align*}
    W_{\bf w_0}(x)=\begin{bmatrix}
        1 & 0 & 0 & 0 & \cdots & 0 & 0\\[2mm]
        e^{x_1} & e^{x_1} & 0 & 0 & \cdots & 0 & e^{x_1} \\[2mm]
        e^{x_2} & 0 & e^{x_2} & 0 & \cdots & 0 & e^{x_2} \\[2mm]
        %e^{x_2} & e^{x_2} & 0 & e^{x_2} & \ddots & \vdots\\[2mm]
        \vdots & \vdots & \vdots & & \ddots & \vdots & \vdots \\
         e^{x_{n}} & 0 & 0 & \cdots & 0 & e^{x_{n}} & e^{x_{n}}\\[2mm]
        e^{-x_{1}} & -e^{-x_{1}} & 0 & \cdots & 0 & 0 & e^{-x_1} \\[2mm]
    \end{bmatrix} 
\end{align*}
Here $W_{\bf w_0}(x)$ is not a lower triangular matrix. However, subtract the 2nd, 3rd, $\cdots$, (n+1)th columns from the last column does transform it into a lower triangular matrix with nonzero diagonal elements. It is easy to see that the determinant is given by $2\prod_{i=2}^ne^{x_i}$. As $\Omega$ is a bounded domain, we hence have the condition \eqref{Wpos} is satisfied.

\medskip
\noindent{\bf Example 3}. As a general formula which contains the above two examples as special cases, we may consider for $x=(x_1, \cdots, x_n)\in\Omega$,
\begin{equation*}
w_0^{(1)}(x) = 1, \ w_0^{(i)}(x) = w_0^{(i)}(x_1,\cdots, x_{i-1}), \ i=2,\cdots, n+1, \ w_0^{(n+2)}(x) = w_0^{(n+2)}(x_1).
\end{equation*}
where the functions $w_0^{(i)}(x)$, $i=2, \cdots, n+2$, satisfy that there exist positive constants $r_1, \cdots, r_{n+1}$ such that
\begin{equation}\label{positivityexample}
\begin{cases}
|\partial_{x_{i-1}}w_0^{(i)}|\geq r_{i-1}>0, \ i=2,\cdots, n+1; \\[2mm]
|\partial_{x_1}w_0^{(1)}\partial_{x_1x_1}w_0^{(n+2)}-\partial_{x_1}w_0^{(n+2)}\partial_{x_1x_1}w_0^{(1)}|\geq r_{n+1}>0.
\end{cases}
\end{equation}
In this case, the matrix $W_{\bf w_0}(x)$ is given by
\begin{align*}
    W_{\bf w_0}(x)=\begin{bmatrix}
        1 & 0 & 0 & 0 & \cdots & 0 & 0\\[2mm]
         w_0^{(2)} & \partial_{x_1}w_0^{(2)} & 0 & 0 & \cdots & 0 & \partial_{x_1x_1}w_0^{(2)} \\[2mm]
        w_0^{(3)} & \partial_{x_1}w_0^{(3)} & \partial_{x_2}w_0^{(3)} & 0 & \cdots & 0 & \Delta w_0^{(3)} \\[2mm]
        %e^{x_2} & e^{x_2} & 0 & e^{x_2} & \ddots & \vdots\\[2mm]
        \vdots & \vdots & \vdots & & \ddots & \vdots & \vdots \\
        w_0^{(n+1)} & \partial_{x_1}w_0^{(n+1)} & \partial_{x_2}w_0^{(n+1)} & \partial_{x_3}w_0^{(n+1)} & \cdots & \partial_{x_n}w_0^{(n+1)} & \Delta w_0^{(n+1)} \\[2mm]
        w_0^{(n+2)} & \partial_{x_1}w_0^{(n+2)} & 0 & 0 & \cdots & 0 & \partial_{x_1x_1}w_0^{(n+2)} \\[2mm]
    \end{bmatrix} 
\end{align*}
Note if we multiply the second row by $\frac{\partial_{x_1}w_0^{(n+2)}}{\partial_{x_1}w_0^{(2)}}$ and subtract it from the last row, we will make the second entry in the last row zero. Then by simple row operations between the 2nd, 3rd, $\cdots$, (n+1)th row and the last row may make the last column all zero except the last entry, which equals to $$\disp\partial_{x_1x_1}w_0^{(n+2)} - \frac{\partial_{x_1x_1}w_0^{(2)}\partial_{x_1}w_0^{(n+2)}}{\partial_{x_1}w_0^{(2)}}.$$
Therefore we arrive at a lower triangular matrix and we may easily compute the determinant as 
\begin{eqnarray*}
& \ & \prod_{i=2}^{n+1}\partial_{x_{i-1}}w_0^{(i)}\left[\partial_{x_1x_1}w_0^{(n+2)} - \frac{\partial_{x_1x_1}w_0^{(2)}\partial_{x_1}w_0^{(n+2)}}{\partial_{x_1}w_0^{(2)}}\right] \\[2mm]
& = & \prod_{i=3}^{n+1}\partial_{x_{i-1}}w_0^{(i)}\left[\partial_{x_1}w_0^{(2)}\partial_{x_1x_1}w_0^{(n+2)} - \partial_{x_1x_1}w_0^{(2)}\partial_{x_1}w_0^{(n+2)}\right]. 
\end{eqnarray*}
The assumption (\ref{positivityexample}) then guarantees the positivity condition \eqref{Wpos} is satisfied.

\medskip
Finally, let us mention some general comments and remarks.

\medskip
(1) It is also possible to set up our inverse problem by assuming Neumann boundary condition $\frac{\partial w}{\partial \nu}$ on $\Sigma^T$ and making measurements of Dirichlet boundary traces of the solution $w$ over $\Sigma_1^T$. This, however, would require more demanding geometrical assumptions on the unobserved portion of the boundary $\Gamma_0$. For example, we may need to assume $\langle Dd,\nu\rangle=0$ on $\Gamma_0$ in the geometrical assumption to account for the Neumann boundary condition \cite{TX2007}. In addition, the more delicate regularity theory for Schr\"{o}dinger eqaution with nonhomogeneous Neumann boundary condition will also need to be invoked. Nevertheless, the main ideas of solving the inverse problem remain the same.

\medskip
(2) An important first step in our proof is to extend the solution of the Schr\"{o}dinger equation (\ref{uq2PDE}) to $[-T,0]$, as denoted in (\ref{extension}). Consequently, we need to assume all the coefficients in the Schr\"{o}dinger equation are real-valued. As such an extension is needed in most of the solution methods based on Carleman estimates, there were very few work on recovering complex coefficients with single or finitely many boundary measurements for the Schr\"{o}dinger equation. The recent work \cite{HKSY2019} showed that it is possible to recover a complex electrical potential from a single Neumann boundary measurement under certain uniformly boundedness assumption on the difference between two electrical potentials. On the other hand, let us also remark that one does not need the coefficients to be real-valued in order to derive Carleman estimates, although certain boundedness conditions like (\ref{LoTFunBound}) on the coefficients will be needed. 

\medskip
(3) The inverse problem formulated in this paper is often referred as the so-called ``single measurement'' type inverse problem. Let us mention that there is another standard ``many measurements'' formulation for inverse problem for the Schr\"{o}dinger equation. In such formulation one seeks to recover coefficients of the Schr\"{o}diner equation from all possible boundary measurements, which are often modeled by the Dirichlet to Neumann map. Although more measurements are required, an advantage in this set up is that one does not need certain positivity assumption on the initial condition in order to solve the inverse problem. In particular, complex-valued coefficients may also be considered. The standard solution methods in the many measurements formulation often involve using the geometric optics solutions \cite{B2017, BDSF2010} or the Boundary Control (BC) method developed by Belishev \cite{Be1987}. For the latter we also refer to the monograph \cite{KKL2000}.

\section*{Appendix: Energy Estimate for the Schr\"{o}dinger equation}

In this appendix we provide a proof for the energy estimate of the general Schr\"{o}dinger equation defined on a Riemannian manifold as in Section 2. This should be a standard estimate, but we did not find any good reference so we decided to include a short proof here. 

\begin{theorem}[Energy Estimate]
For the Schr\"{o}dinger equation \eqref{genPDE} with the energy level differential term $G$ defined as in (\ref{ManLoTF}) and (\ref{LoTFunBound}), and  forcing term $F\in H^1(-T,T; L^2(\Omega))$. Suppose we have the homogeneous Dirichlet boundary condition $w=0$ on $\Sigma=[-T,T]\times\Gamma$ for some fixed positive $T>0$, and the initial condition $w_0\in H^1_0(\Omega)$. Then we have
\begin{equation}\label{EE} 
\mathbb{E}_w(t) \leq C \left(\|w_0\|^2_{H^1_0(\Omega)}+\|F\|^2_{H^{1}(-T,T;L^2(\Omega))}\right), \ \forall t\in[-T,T]
\end{equation}
where $C$ depends on $\Omega,\ T,\ \vv{P}_1$ and $P_0$.
\end{theorem}
\begin{proof}
First, let us consider the case where $t\in[0,T]$.
We start with the Schr\"{o}dinger equation \eqref{genPDE}. Multiply \eqref{genPDE} by $\ol{w}$, integrate over $Q^t=[0,t]\times\Omega$, and take the imaginary part, we get
\begin{align*}&
        \int_{Q^t}
         \text{Im}\left[iw_t\ol{w}\right]
        +\text{Im}\left[\Delta_g w\ol{w}\right]
        +\text{Im}\left[\left\langle\vv{P}_1, D w\right\rangle\ol{w}\right]
        +\text{Im}\left[P_0|w|^2\right]
        dQ^t
        % \\&\hspace{2cm}
        =
        \int_{Q^t} \text{Im}\left[F\ol{w}\right]dQ^t.
\end{align*}
Apply Green's identity along with eliminating any purely real part or zero boundary terms, we obtain
\begin{align*}&
            \int_{Q^t} 
         \text{Im}\left[\frac i2 \pt |w|^2 \right]
        +\text{Im}\left[\left\langle\vv{P}_1, D w\right\rangle\ol{w}\right]
         dQ^t
         =\int_{Q^t} \text{Im}\left[F\ol{w}\right]dQ^t.
\end{align*}
Then by Cauchy--Schwartz inequality and (\ref{LoTFunBound}), we arrive at
\begin{align*}&
            \int_{\Omega} 
         |w(t,x)|^2
         d\Omega
        \leq
        \|w_0\|^2_{L^2(\Omega)}
        +C\int_{Q^t} \left|F\right|^2 + \left|D w\right|^2 + \left|w\right|^2dQ^t.
\end{align*}
Applying Gronwall's inequality we then get
\begin{align}
        \int_{\Omega} 
         |w(t,x)|^2
         d\Omega
        \leq
        Ce^{Ct}\left(
        \|w_0\|^2_{L^2(\Omega)}
        +\|F\|^2_{L^2(Q^t)}
        +\|Dw\|^2_{L^2(Q^t)}
         \right)
        .\label{wBound}
\end{align}
Similarly, now we multiply  \eqref{genPDE}  by $\ol{w}_t$, integrate over $Q^t$ and take the real part we have
\begin{align}&
    \int_{Q^t} 
     \text{Re}\left[i|w_t|^2\right]
    +\text{Re}\left[\Delta_g w\ol{w}_t\right]
    +\text{Re}\left[\left\langle \vv{P}_1, D w\right\rangle\ol{w}_t\right]
    +\text{Re}\left[P_0 w\ol{w}_t\right]dQ^t
    % \\&\hspace{2cm}
    =
    \int_{Q^t} \text{Re}\left[F\ol{w}_t\right]dQ^t \label{52}
\end{align}
Apply Green's identity, eliminating the purely imaginary and zero boundary terms, and use the equation \eqref{genPDE} to rewrite the $\ol{w}_t$ term in the third integral term on the left-hand side, we get
\begin{align}&
     \int_{Q^t} 
    -\frac12\pt|D w|^2
    -\text{Re}\left[i\left\langle \vv{P}_1, D w\right\rangle \Delta_g\ol{w}\right]
    -\text{Re}\left[i\left\langle \vv{P}_1, D w\right\rangle P_0\ol{w}\right]\nonumber
    \\&\hspace{2cm}
    +\text{Re}\left[i\left\langle \vv{P}_1, D w\right\rangle \ol{F}\right]
    +P_0\frac12\pt|w|^2dQ^t
    =
    -\int_{Q^t} \text{Re}\left[F_t\ol{w}\right]dQ^t.\label{RealEnergyPause1}
\end{align}
Now for the second term on the left-hand side of (\ref{RealEnergyPause1}), by a combination of Green's identity and the following identity (see, for instance, \cite[Lemma~2.1]{Y1999})
$$\left\langle D\overline{w},D\left\langle H,Dw\right\rangle\right\rangle=\left\langle D_{D\overline{w}}H,Dw\right\rangle+\frac12\left\langle H,D\left(\left|Dw\right|^2\right)\right\rangle$$
%$$\langle D \ol{w}, D\langle h, D w\rangle\rangle=\langle D_{\ol{w}}h,Dw\rangle +\frac12\langle h, D(|D w|^2)\rangle,$$
where $H:\mathcal{M}\to T_x\mathcal{M}$, and eliminate purely imaginary terms, we get
\begin{align}
    \int_\Omega\text{Re}\left[i\left\langle \vv{P}_1, D w\right\rangle \Delta_g \ol{w}\right]
    &=    -\int_\Omega\text{Re}\left[i\langle D_{D\ol{w}}\vv{P}_1, Dw\rangle\right],\label{LapGrad}
\end{align}
where $\langle D_{X}\vv{P}_1, DY\rangle$ is the Hessian of $\vv{P}_1$ evaluated at $X, Y\in T_x\mathcal{M}$. Applying \eqref{LapGrad} to \eqref{RealEnergyPause1}, we have
\begin{eqnarray*}
     & \ &    \int_{\Omega}|D w(t,x)|^2d\Omega \\[2mm]
  & = &\int_{\Omega}|D w_0|^2+P_0|w(t,x)|^2-P_0|w_0|^2d\Omega-\int_{Q^t} 2\text{Re}\left[i\langle D_{D\ol{w}}(\vv{P}_1), Dw\rangle\right]dQ^t \\[2mm]
   & \ & +\int_{Q^t} -2\text{Re}\left[i\left\langle \vv{P}_1, D w\right\rangle P_0\ol{w}\right]+2\text{Re}\left[i\left\langle \vv{P}_1, D w\right\rangle \ol{F}\right]+2\text{Re}\left[F_t\ol{w}\right]dQ^t.
\end{eqnarray*}
Then by \eqref{LoTFunBound}, \eqref{wBound} and the Poincar\'e inequality, we have
\begin{align*}&
     \int_{\Omega}|D w(t,x)|^2d\Omega\leq Ce^{Ct}\left(\|w_0\|^2_{H^1_0(\Omega)}+\|F\|^2_{H^1(0,t;L^2(\Omega))}+\int_{Q^t} |D w|^2 dQ^t\right).
\end{align*}
With Gronwall's inequality we then get
\begin{align}&
     \int_{\Omega}|D w(t,x)|^2d\Omega\leq  Ce^{Ct}\exp(Ce^{Ct})\left(\|w_0\|^2_{H^1_0(\Omega)}+\|F\|^2_{H^1(0,t;L^2(\Omega))}\right).\label{GwBound}
\end{align}
Using \eqref{wBound} and \eqref{GwBound} on $\mathbb{E}_w(t)$ and  we have,
\begin{align*}
    \mathbb{E}_w(t)
    =&
    \int_\Omega \left|w(t,x)\right|^2 + \left|D w(t,x)\right|^2 d\Omega
        \\\leq&
    Ce^{Ct}\exp(Ce^{Ct})
    \left(
    \|w_0\|^2_{H^1(\Omega)}
    +t\|F\|^2_{H^1(0,t;L^2(\Omega))}
    \right).
    % \\&
    % +Ce^{Ct}\exp(Ce^{Ct})\left(\|w_0\|^2_{H^1(\Omega)}
    % +\|F\|^2_{H^1(0,t;L^2(\Omega))}\right)
\end{align*}
This implies the desired estimate (\ref{EE}) with $t\in [0,T]$. In an almost identical fashion we can also get the same estimate for $t\in[-T,0]$.
\end{proof}
\begin{rem}
The regularity on $F$ can be replace by $L^2(-T,T; H^{1}(\Omega))$ for which the corresponding $\|F\|^2_{L^2(-T,T; H^{1}(\Omega))}$ term will show up in the estimate (\ref{EE}) instead of $\|F\|^2_{H^{1}(-T,T;L^2(\Omega))}$.
This may be seen from (\ref{52}) if we also use the equation \eqref{genPDE} to rewrite the $\ol{w}_t$ term on the right-hand side, instead of doing integration by parts with $t$.
\end{rem}

%\bigskip
%{\bf Acknowledgements}

\end{document}